%% file: Reducible-201605-arxiv.tex
\documentclass[11pt]{amsart}

\usepackage{amsfonts, amstext, amsmath, amscd, amssymb}
\usepackage{enumerate}
\usepackage{booktabs}
\usepackage{import}
\usepackage{epsfig, graphicx, psfrag}
\usepackage{color}

\numberwithin{equation}{section}
\theoremstyle{plain}
\newtheorem{theorem}{Theorem}
\numberwithin{theorem}{section}

\newtheorem{lemma}[theorem]{Lemma}
\newtheorem{conjecture}[theorem]{Conjecture}
\newtheorem{prop}[theorem]{Proposition}

\newtheorem*{namedtheorem}{\theoremname}
\newcommand{\theoremname}{testing}
\newenvironment{named}[1]{\renewcommand{\theoremname}{#1}\begin{namedtheorem}}{\end{namedtheorem}}

\theoremstyle{definition}
\newtheorem{define}[theorem]{Definition}
\newtheorem{remark}[theorem]{Remark}
\newtheorem{question}[theorem]{Question}

\usepackage[
pdfborderstyle={},
pdfborder={0 0 0},
pagebackref,
pdftex]{hyperref}

\renewcommand*{\backref}[1]{}
\renewcommand*{\backrefalt}[4]{
  \ifcase #1 %
   [No citations.]%
  \or
   [#2]%
  \else
   [#2]%
  \fi
}

\AtBeginDocument{%
   \def\MR#1{}
}

\catcode `\@=11
\long\def\@savemarbox#1#2{\global\setbox#1\vtop{\hsize\marginparwidth 
  \@parboxrestore\tiny\raggedright #2}}
\catcode`\@=12

\usepackage{marginnote}

\newcommand{\HH}{{\mathbb{H}}}
\newcommand{\RR}{{\mathbb{R}}}

\newcommand{\CC}{{\mathbb{C}}}

\newcommand{\bdy}{{\partial}}

\newcommand{\PSL}{{\rm PSL}}

\def\co{\colon\thinspace}

\definecolor{purple}{rgb}{0.63, 0.36, 0.94}

\begin{document}
\title{Geometry of planar surfaces and exceptional fillings}

\author{Neil R. Hoffman}
\address{School of Mathematics and Statistics, University of Melbourne, Parkville, VIC, 3010, Australia}

\author{Jessica S. Purcell}
\address{School of Mathematical Sciences, 9 Rainforest Walk, Monash University, VIC, 3800, Australia}

\begin{abstract}
If a hyperbolic 3--manifold admits an exceptional Dehn filling, then the length of the slope of that Dehn filling is known to be at most six. However, the bound of six appears to be sharp only in the toroidal case. In this paper, we investigate slope lengths of other exceptional fillings. We  construct hyperbolic 3--manifolds that have the longest known slopes for reducible fillings. As an intermediate step, we show that the problem of finding the longest such slope is equivalent to a problem on the maximal density horoball packings of planar surfaces, which should be of independent interest. We also discuss lengths of slopes of other exceptional Dehn fillings, and prove that six is not realized by a slope corresponding to a small Seifert fibered space filling.
\end{abstract}

\maketitle

\section{Introduction}

By the $2\pi$--Theorem of Thurston and Gromov, any Dehn filling of a cusped hyperbolic 3--manifold along slopes of length more than $2\pi$ results in a manifold with a negatively curved metric (see \cite{bleilerhodgson}). By the Geometrization Theorem, such a negatively curved 3--manifold must admit a complete hyperbolic structure. Thus if a Dehn filling does not yield a hyperbolic manifold, i.e.\ if it is an \emph{exceptional filling}, then at least one of the slopes of the filling must have length at most $2\pi$.
Independently, Agol \cite{agol:bounds} and Lackenby \cite{lackenby:word} proved the $6$--Theorem, which lowered the upper bound on the length of an exceptional filling to six.

The Geometrization Theorem implies that any non-hyperbolic 3--manifold is either reducible, Seifert fibered, or toroidal. Agol showed that the 6--Theorem gives a sharp bound on slope length of exceptional fillings by exhibiting a hyperbolic manifold with a toroidal filling of length exactly six. In addition, Adams \emph{et al} \cite{adams:TotallyGeodesic} found an infinite family of hyperbolic knots $K_i$ such that each $S^3-K_i$ admits a toroidal filling of length six. In addition to these examples of toroidal fillings, obtaining bounds on the meridians of knot complements has also been a subject of inquiry. It is conjectured that four is the maximal length of a meridian of a knot complement. This conjectural bound is asymptotically sharp as observed in a number of places (see for example \cite[$\S$7]{agol:bounds} and \cite[Theorem 4.3]{purcell:slope}).  As far as we are aware, no one has previously considered the natural question of determining upper bounds on slope lengths for other types of exceptional fillings.

In this paper, we investigate that question. Our primary focus for the first part of the paper is \emph{reducible} Dehn fillings, i.e.\ those that produce a 3--manifold with an embedded essential 2--sphere. For such fillings, the bound of six does not appear to be sharp. We wish to determine the longest possible slope giving a reducible filling of a hyperbolic 3--manifold. Towards this goal, we present the following theorem which proves the existence of reducible fillings that asymptotically approach length $10/\sqrt{3}> 5.77$. 
 
\begin{theorem}\label{thm:main}
 For every $\epsilon >0$, there exists a hyperbolic 3--manifold $M$ and slopes $s_1, \dots, s_n$ such that the Dehn filled manifold $M(s_1, \dots, s_n)$ is reducible and $s_i$ has length at least $\frac{10}{\sqrt{3}}-\epsilon$. 
\end{theorem}

Note that if a hyperbolic 3--manifold $M$ admits a reducible filling, then it contains an essential punctured 2--sphere. We may pleat this 2--sphere in $M$ to obtain a hyperbolic structure on a planar surface. One question that can be asked is, which hyperbolic planar surfaces arise in this manner? 
The cusp neighborhoods of $M$ induce cusp neighborhoods of $S$, and the lengths of curves tracing out the boundary of these cusp neighborhoods of $S$ must be at least as long as the exceptional slopes of $M$ (see, for example \cite[Lemma~2.5]{futer-schleimer}). Thus a lower bound on slope length of $M$ gives a lower bound on slope length of $S$.

On the other hand, in this paper we show:

\begin{named}{Theorem~\ref{Thm:ReducibleFilling}}
Let $R$ be a hyperbolic structure on a planar surface with a fixed cusp neighborhood $H$. Then for any $\epsilon>0$, in any neighborhood of $R$ in its moduli space, there exists a hyperbolic surface $R'$, and a hyperbolic 3--manifold $M$ such that:
\begin{enumerate}
\item $R'$ is isometric to a totally geodesic surface embedded in $M$.
\item There exists a cusp neighborhood of $M$ such that each boundary slope of $R'$ in $M$ has length within $\epsilon$ of the corresponding length on $\partial H$ in $R$.
\item Dehn filling $M$ along the boundary slopes of $R'$ results in a reducible manifold.
\end{enumerate}
\end{named}

Note that this theorem and the discussion above imply that finding long reducible slopes can be reduced to finding hyperbolic structures on planar surfaces with long cusp lengths. For if $R$ is a hyperbolic planar surface with cusps of length $L_1, \dots, L_n$, the theorem implies that for any $\epsilon>0$, there is a 3--manifold $M$ admitting a reducible filling with slope lengths at least $L_1-\epsilon, \dots, L_n-\epsilon$, respectively. On the other hand, the discussion above implies that the slope lengths for this reducible filling are at most the lengths of the cusps of $R'$, which are at most $L_1+\epsilon, \dots, L_n+\epsilon$, respectively. Thus for any hyperbolic structure on a planar surface, there is a sequence of 3--manifolds with reducible fillings whose slope lengths approach those of $R$.
Hence finding long reducible slopes becomes a question of finding hyperbolic structures on planar surfaces, and in particular, finding a horoball packing of a planar surface that maximizes the minimal area horoball.

Along these lines, we submit the following conjecture, which if true would show that our examples in Theorem~\ref{thm:main} produce the (asymptotically) longest possible reducible fillings.
 
\begin{conjecture}
Let $F$ be a planar surface admitting a hyperbolic structure. For any horoball packing of $F$, there is at least one horoball of area less than $10/\sqrt{3}$. 
\end{conjecture}  

We conclude the paper with discussion of the Seifert fibered and finite exceptional filling cases. Similar to the reducible case, for such fillings the longest possible slope yielding this type of manifold is still unknown. We discuss the longest known examples of these, and note that experimental and theoretical evidence points to bounds strictly less than six in these cases. 
For the Seifert case, we prove the following. 

\begin{named}{Theorem~\ref{thm:SSFbound}}
  Let $M$ be a hyperbolic manifold with one cusp, and let $s$ be a slope such that $M(s)$ is a small Seifert fibered space with infinite fundamental group.  Then the length of $s$ is strictly less than six.
\end{named}

\subsection{Acknowledgements}
The first author is partially supported by ARC Discovery Grant DP130103694. The second author is partially supported by NSF Grant DMS-1252687 and DMS-1128155, and thanks the University of Melbourne for hosting her while working on this project.
We thank Craig Hodgson and Ian Agol for helpful discussions. We also thank the referee for very helpful comments.

\section{Existence of reducible fillings}

At most countably many hyperbolic structures on a planar surface can be embedded as a totally geodesic surface in a finite volume hyperbolic 3--manifold, because there are only countably many finite volume hyperbolic 3--manifolds. However, the following theorem shows that the set of hyperbolic structures on planar surfaces that can be embedded is dense in the moduli space of the surface.

\begin{theorem}\label{Thm:Cusped}
Let $R$ be a hyperbolic surface with finite genus and with at least one but at most finitely many cusps.  Then in any neighborhood of $R$ in its moduli space, there exists a hyperbolic surface $R'$ with the following properties.
\begin{enumerate}
\item There exists a finite volume cusped hyperbolic 3--manifold $M$ that contains an embedded surface isometric to $R'$.
\item For any $\epsilon>0$ and any embedded cusp neighborhood $H$ of $R$, we may take $M$ such that there exists an embedded cusp neighborhood of $M$ for which each boundary slope of $R'$ in $M$ has length within $\epsilon$ of the corresponding length on $\partial H$ in $R$. 
\item Each cusp of $R'$ is embedded in a distinct cusp of $M$. 
\end{enumerate}
\end{theorem}

The proof of Theorem~\ref{Thm:Cusped} constructs $M$ by appealing to circle packing techniques of Brooks, \cite{Brooks:DeformationSchottky} and \cite{Brooks:CirclePackings}. This argument is similar to ideas of Fujii \cite{Fujii:Deformations} that state that closed surfaces satisfy a similar property. 

\begin{define}\label{def:circlePacking}
A \emph{circle packing} of a polygonal region is a collection of circles with disjoint interiors embedded in the closure of the region, such that all circles are tangent to other circles, and the exterior of the union of circles is a disjoint union of curvilinear triangles. Given a collection of circles, an \emph{interstice} is a curvilinear polygon in the complement of the interiors of the circles, with boundary made up of pieces of the circles. Thus a circle packing is a collection of circles for which all interstices are triangles.
\end{define}


Given a hyperbolic surface $R$, begin by taking a fundamental domain for $R$ in $\HH^2$ that is a finite ideal polygon $P$, with side pairings given by isometries of $\HH^2$. For example, $P$ can be obtained by fixing a cusp neighborhood for cusps of $R$, and taking the corresponding canonical decomposition of $R$ as in Epstein--Penner \cite{EpsteinPenner}.

The universal cover of $R$ is $\HH^2$, which we view as the equatorial plane of $\HH^3$ under the ball model.
For each edge of $P$, take a geodesic plane in $\HH^3$ whose intersection with $\HH^2$ is that edge, and such that the plane is orthogonal to the copy of $\HH^2$.  Notice that a side--pairing isometry of this edge will take the geodesic plane in $\HH^3$ to a geodesic plane meeting $\HH^2$ orthogonally, with intersection another edge of $P$. Thus the side--pairing isometries, which generate the fundamental group of $R$ by the Poincar{\'e} polyhedron theorem, act on this collection of geodesic planes in $\HH^3$. The boundaries of these geodesic planes give a collection of circles on $\bdy_\infty \HH^3 = S^2_\infty$ that are tangent in pairs, all orthogonal to the equator of $S^2_\infty$.  We color these circles \emph{blue}. The collection of face pairing isometries $\Gamma\subset{\rm{Isom}}^+(\HH^3)$ forms a Fuchsian group that preserves blue circles. 
The quotient of the action of $\Gamma$ on $\HH^3$ gives an infinite volume Fuchsian manifold, with rank--1 cusps.

\begin{lemma}\label{lemma:CirclePack}
In any neighborhood of $R$ in its moduli space, there exists a hyperbolic surface $R'$ for which the infinite volume Fuchsian manifold as above admits a circle packing on its boundary at infinity.
\end{lemma}

\begin{proof}
We build the circle packing in four steps, beginning with the Fuchsian manifold associated to $R$ given by face pairings of blue hemispheres.
  
\smallskip

\textbf{Step 1: Select circles tangent to ideal vertices.}
Select an ideal vertex $v$ of $P$ and choose a small circle on $S^2_\infty$ tangent to the equatorial plane at $v$. Also take its image under reflection in the equatorial plane. The face pairings of $P$ take these circles to a finite collection of circles, all tangent to the equatorial plane at ideal vertices of $P$. By choosing the initial radius small enough, these circles will be mapped to disjoint circles under $\Gamma$. Add such circles for each ideal vertex of $P$, and color them \emph{green}. Note the green circles are preserved under the action of $\Gamma$. 

\smallskip

\textbf{Step 2: Select circles on edges.}
Next, select an ideal edge of $P$. This corresponds to half of a blue circle in the northern hemisphere of $S^2_\infty$. Choose a finite collection of pairwise tangent circles orthogonal to this blue circle with the initial and final circles tangent to the green circles at the endpoints of the corresponding ideal edge. One of the face pairings of $\Gamma$ takes this collection of circles to a collection of pairwise tangent circles on another blue half-circle, identified to the original. See for example Figure~\ref{fig:blueRedPacking}(a). Again we may ensure the circles have disjoint interiors, and again reflect through the equatorial plane. Repeat this process for each ideal edge, and color the resulting circles \emph{green}. When finished, we have a collection of pairwise tangent green circles with disjoint interiors, orthogonal to blue circles. The collection of blue circles, green circles, and the equatorial plane is preserved under the action of $\Gamma$.

\begin{figure}
  \centering
  \begin{tabular}{ccccc}
\includegraphics[height=1.2in]{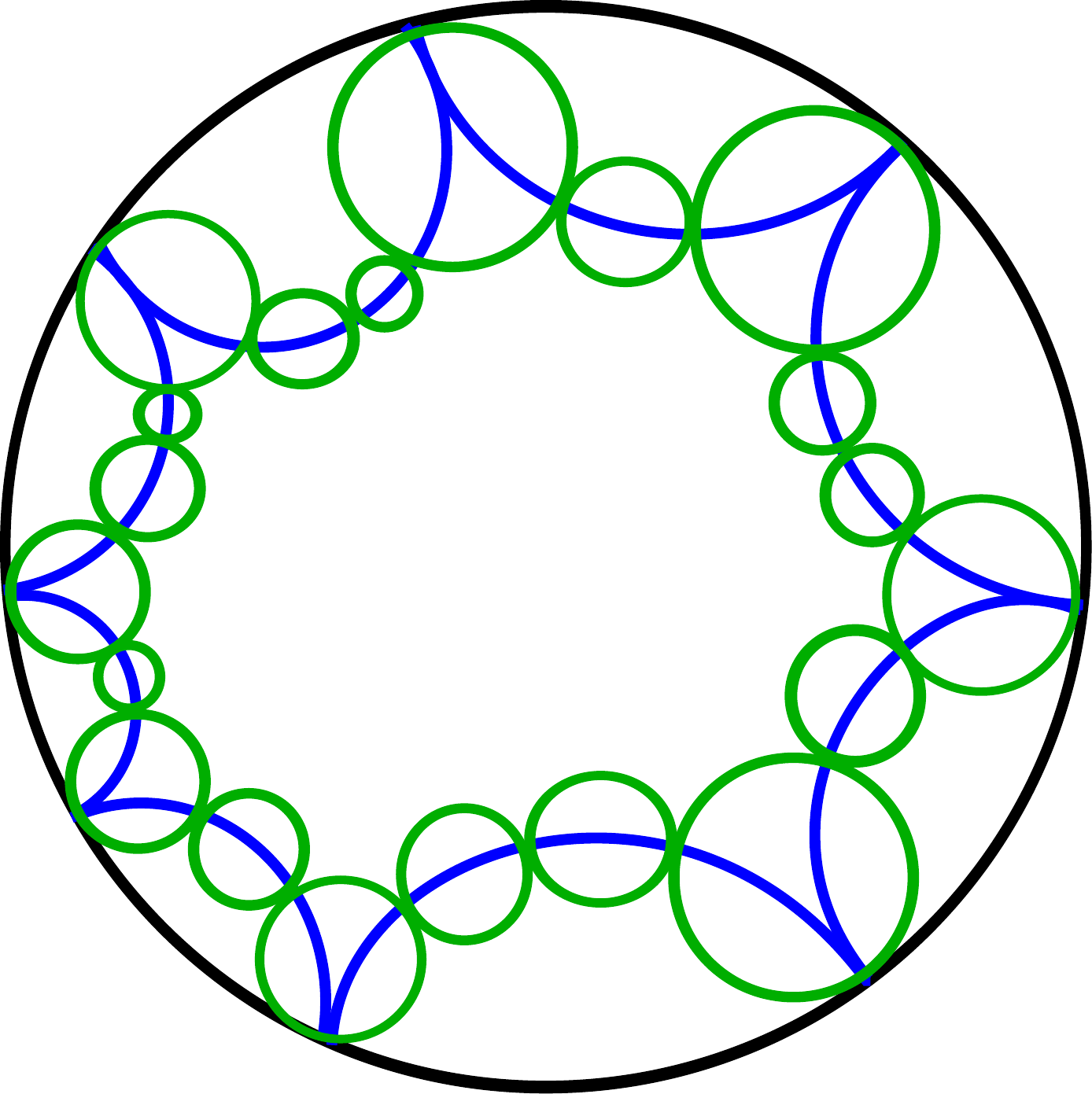} & &
\includegraphics[height=1.2in]{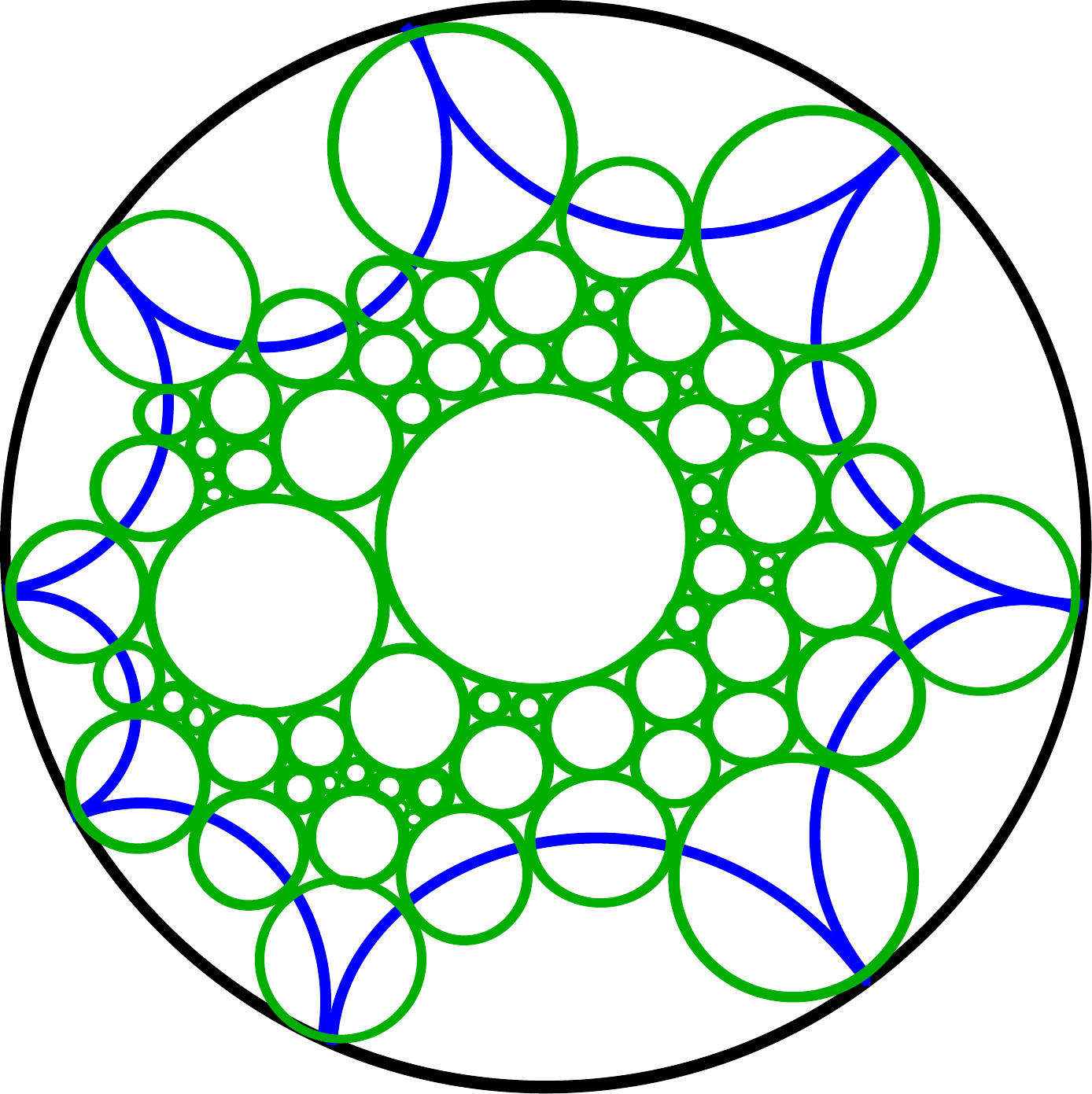} & &
\includegraphics[height=1.2in]{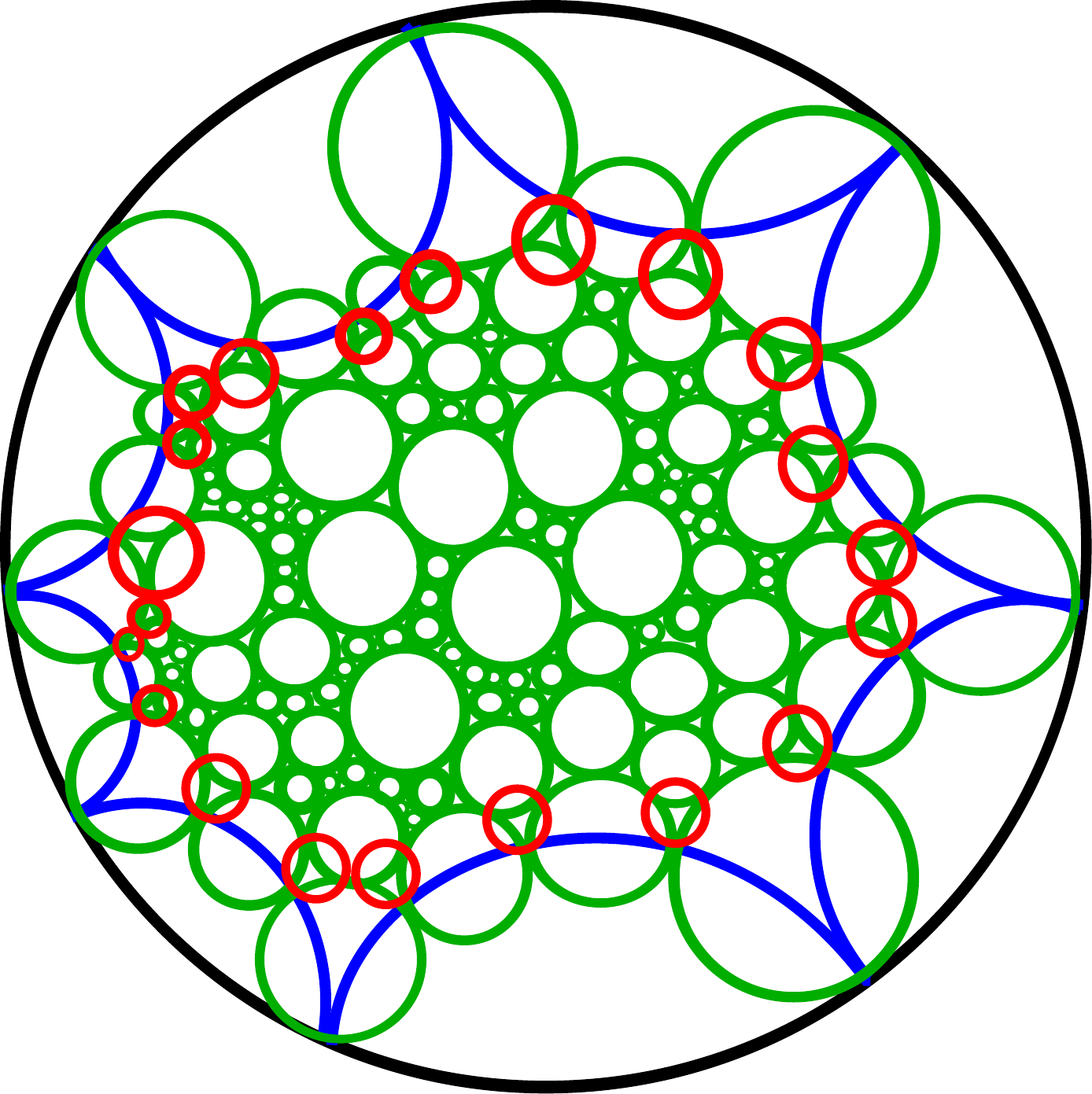}\\
(a) & \hspace{.2in} & (b) &\hspace{.2in} & (c)\\
  \end{tabular}
  \caption{(a) Packing green circles along edges. (b) A packing of circles in a finite ideal polygon $P$, with semi-circles coming from face pairings (blue) and other circles (green) in the interior. (c) The packing from before together with some of the dual circles (red).}
  \label{fig:blueRedPacking}
\end{figure}

\smallskip

\textbf{Step 3: Select circles in interior.}
Now in the region of the northern hemisphere of $S^2_\infty$ disjoint from all interiors of blue circles, add finitely many circles with disjoint interiors, tangent to existing green circles. 
The result will not necessarily be a circle packing, with all triangular interstices. However, we may arrange such that all interstices are triangles and quadrilaterals, with only a finite number of quadrilateral interstices. Continue to color these circles \emph{green}, and reflect them across the equator to obtain a collection of circles in $S^2_\infty$, colored
green and blue, that meet in triangular interstices and a finite number of quadrilateral interstices. As in \cite{Brooks:DeformationSchottky}, the quadrilaterals are parameterized by $\RR$, with rational numbers corresponding to quadrilaterals which admit an actual circle packing. 

\smallskip

\textbf{Step 4: Deform to a circle packing.}
Again following Brooks, for each green circle, adjoin to $\Gamma$ the reflection in that green circle. Also add reflection in the equator. This forms a new group $\Gamma'$, which is still discrete by the Poincar{\'e} polyhedron theorem. Then as in \cite{Brooks:CirclePackings}, there exists a quasi-conformal deformation of $\Gamma'$ by an arbitrarily small amount $\delta$ to a new group $\Gamma_\delta'$, taking reflections in circles to reflections in circles, for which all quadrilaterals have rational parameters. That is, the result admits a finite circle packing, by adding finitely many circles to the quadrilateral interstices. 
Let $\Gamma_\delta$ denote the image of $\Gamma$ under this deformation. It is still a Fuchsian group, because $\Gamma_\delta'$ preserves the plane corresponding to the equator, by a reflection. Then $R' = \HH^2/\Gamma_\delta$ is a quasi-conformal deformation of $R$ that is $\delta$-close to $R$, and $R'$ satisfies the conclusions of the lemma. 
\end{proof}

Figure~\ref{fig:blueRedPacking}(b) shows a finite ideal polygon with a circle packing by green circles.

Note that for the infinite volume Fuchsian manifold of Lemma~\ref{lemma:CirclePack}, distinct cusps of $R'$ are embedded in distinct cusps of the Fuchsian manifold.

To obtain the full result of Theorem~\ref{Thm:Cusped}, we need to add more structure to our argument, namely, we want to include the cusp neighborhood of $R$. To do so, we use a bit of extra bookkeeping, namely the \emph{decorated} moduli space. The decorated moduli space can be described locally by what are commonly known as \emph{Penner coordinates} on the decorated Teichm\"uller space, defined (but not named as such) in Penner~\cite{penner1987decorated}. We will not need a precise description of these coordinates. Instead, we will appeal to the following two properties.
\begin{enumerate}
\item The lift of a point in the moduli space of the surface $R$ to the \emph{decorated} moduli space is homeomorphic to $\RR^c$, where $c$ is the number of cusps of $R$. (The coordinates of $\RR^c$ can be viewed as measuring signed distance between cusps.)
\item If $R'$ is a point in a neighborhood of (undecorated) moduli space of $R$ and $H$ is an embedded cusp of $R$, then there exists $(R',H')$ in the decorated moduli space of $(R,H)$ such that cusp lengths associated to $H'$ are arbitrarily close to those of $H$, and $H'$ is an embedded cusp neighborhood.
\end{enumerate}

\begin{lemma}\label{lemma:CirclePackDisjointH}
  Let $(R,H)$ be a surface and its embedded cusp neighborhood. Then in any neighborhood of $(R,H)$ in its decorated moduli space, there exists a point $(R',H')$ such that Lemma~\ref{lemma:CirclePack} holds for $R'$, and moreover $H'$ is an embedded cusp of $R'$ whose lifts are horoballs, each of which meets a hemisphere bounded by a circle in the circle packing of Lemma~\ref{lemma:CirclePack} only when that circle is tangent to the center of the horoball. In this case, the intersection of the horoball and the hemisphere is a noncompact region of the hemisphere.
\end{lemma}

\begin{proof}
Given $R$ and $H$, the cusp neighborhood $H$ lifts to a disjoint collection of horoballs in the Fuchsian manifold $\HH^3/\Gamma$. We let $\widetilde{H}$ denote finitely many horoballs, one at each ideal vertex of $P$.   
We prove the result by stepping back through the proof of Lemma~\ref{lemma:CirclePack}, ensuring that we may choose circles at each step to be disjoint from the collection $\widetilde{H}$, or to meet them only in noncompact regions if they share a point on the boundary at infinity. 

In the first step of the proof of Lemma~\ref{lemma:CirclePack}, we choose green circles tangent to vertices of $P$. To see that these can be chosen as in the statement of the lemma, consider the (finitely many) vertices of $P$ identified by face pairings. As we shrink any circle tangent to such a vertex, its images under face pairings also shrink. Thus we may shrink any such circle until the hyperplane bounded by it meets only the horoball at the ideal vertex at which it is tangent, and similarly for its translates under face pairings.

Now view the ball model of $\HH^3$ as the unit ball in $\RR^3$. Let $C$ denote the complement on $S^2$ of the interiors of disks bounded by blue circles and the green circles tangent to the equator, selected in the last paragraph. Then $C$ is compact, as is the collection of finite (closed) horoballs $\widetilde{H}$, and these two sets are disjoint. Thus the Euclidean distance $d(C,\widetilde{H})$ is some positive value $d_0>0$. Let $r_0=d_0/2$. Then any Euclidean hemisphere with center in $C$ of radius $r_0$ will have Euclidean distance at least $d_0/2$ from $\widetilde{H}$, hence its hyperbolic distance will also be bounded away from zero.

Use this to complete the second and third steps of selecting the circle packing. We may ensure that all circles selected in these steps have radius less than $r_0$; since their centers lie in $C$, they are disjoint $\widetilde{H}$. By compactness, we may also ensure that the triangle and quadrilateral interstices are shaped such that if we were to add any additional circles to the collection, those circles must have radius less than $r_0/2$, so adding any finite collection of circles to the existing collection yields new hemispheres disjoint from $\widetilde{H}$. 

Followng the fourth step, we next adjust the hyperbolic structure to $R'$ admitting a circle packing. Further, for any neighborhood of $(R',H')$ in decorated moduli space, we may obtain an embedded cusp neighborhood $H'$ of $R'$ such that $(R',H')$ lies in that neighborhood. Let $\widetilde{H}'$ denote the finitely many horoballs that are lifts of the horoball neighborhood $H'$ with centers on points corresponding to deformed points of $P$.

Before deforming $R$ to $R'$, the distance between horoballs $\widetilde{H}$ and those green hemispheres that are not tangent to the equator is bounded away from $0$, say by $d>0$, so we can ensure our deformation is small enough that the distance from new hemispheres to $\widetilde{H}'$ is still at least $d/2>0$. Moreover, we can ensure that after the deformation, interstices are still sufficiently small that any added circles correspond to hemispheres disjoint $\widetilde{H}'$. Thus when we complete to a circle packing, all green circles satisfy the conclusion of the lemma. 
\end{proof}

Theorem~\ref{Thm:Cusped} requires a finite volume manifold $M$. So far, Lemmas~\ref{lemma:CirclePack} and~\ref{lemma:CirclePackDisjointH} provide us with an infinite volume Fuchsian manifold admitting a circle packing (by green circles) on its conformal boundary. To form the finite volume 3--manifold, add \emph{red} circles: There exists a unique red circle that meets each vertex of each triangular interstice between green circles.
Figure~\ref{fig:blueRedPacking}(c) shows some of the dual red circles for the packing example from that figure.
The red and green circles meet orthogonally.

\begin{lemma}\label{lemma:RedCircles}
Red circles formed as above are disjoint from $\widetilde{H}'$. 
\end{lemma}

\begin{proof}
Suppose by way of contradiction that some red circle meets a horoball $H_0$ of $\widetilde{H}'$. The red circle is determined by a triangular interstice, with three sides determined by green circles. At most one of these circles is tangent to the equator, so at most one meets the center of a horoball. Thus two green circles disjoint from $\widetilde{H}'$ meet in a vertex $v$ of the interstice. 
Apply an isometry taking $\HH^3$ to the upper half space model, taking $v$ to infinity. Then we have two green hemispheres that are disjoint from $\widetilde{H}'$ mapping to parallel vertical planes. Ensure one maps to the vertical plane parallel to the real line through the point $-i$ in $\CC$, and one maps to the parallel vertical plane through $i$, and the third has boundary mapping to the unit circle in $\CC$. The horoball $H_0$ maps to a horoball disjoint from all three green circles, except possibly centered at a point on the unit circle. In any case, its center will be on $\CC$ in the infinite strip between the two parallel vertical planes. Finally, the red circle maps to the unique hyperplane meeting the three points of tangency of the green; this is the vertical plane with boundary on the imaginary axis of $\CC$.

By assumption, the image of $H_0$ intersects the red vertical plane. However, it cannot meet any of the green planes. Thus its Euclidean radius is strictly less than the imaginary coordinate of its center. Since its center lies no closer to the imaginary axis than the unit circle, it is impossible that it intersects the red plane while remaining disjoint from the two green vertical planes. This is a contradiction. 
\end{proof}

\begin{proof}[Proof of Theorem~\ref{Thm:Cusped}]
Starting with $(R,H)$, obtain $(R',H')$ and its circle packing as in Lemma~\ref{lemma:CirclePackDisjointH}, with $\Gamma'$ denoting the group of isometries of $\HH^3$ giving face pairings of the blue hemispheres associated with $R'$. Adjoin to $\Gamma'$ reflections in the green circles and reflections in the equatorial plane. Finally, adjoin reflections in the red planes. The quotient of $\HH^3$ by this group gives a finite volume hyperbolic 3--manifold with fundamental domain cut out by red, blue, and green hemispheres in $\HH^3$. It contains an embedded surface isometric to $R'$, namely the image of the equatorial plane.

By construction, each cusp of $R'$ is embedded in a distinct cusp of $M$. To finish the proof it suffices to show that the embedded cusp neighborhood $H'$, which has lengths within $\epsilon$ of those corresponding to $H$, lifts to an \emph{embedded} neighborhood of (some of) the cusps of $M$. 

For suppose lifts of $H'$ are not embedded. Since $H'$ is embedded in $R'$, horoballs $\widetilde{H}'$ with centers on the equator are all disjoint. Thus some horoball $H_1$ must intersect a horoball $H_2$ in the fundamental domain for $M$, and at least one of those horoballs, say $H_2$, does not have its center on the equator. But then $H_2$ cannot have its center at any of the ideal vertices of the fundamental domain. Hence $H_2$ must intersect one of the red, green, or blue faces of the fundamental domain in a compact region. Because $H_2$ is a translate of some horoball of $\widetilde{H}'$ under the fundamental group of $M$, it follows that one of the horoballs of $\widetilde{H}'$ meets a red, green, or blue hemisphere in a compact region. This contradicts our choice of circle packings.
\end{proof}

Theorem~\ref{Thm:ReducibleFilling} from the introduction is an immediate consequence of the above theorem with $R$ a planar surface. 

\begin{theorem}\label{Thm:ReducibleFilling}
Let $R$ be a hyperbolic structure on a planar surface with a fixed cusp neighborhood $H$. Then for any $\epsilon>0$, in any neighborhood of $R$ in its moduli space, there exists a hyperbolic surface $R'$, and a hyperbolic 3--manifold $M$ such that:
\begin{enumerate}
\item $R'$ is isometric to a totally geodesic surface embedded in $M$.
\item There exists a cusp neighborhood of $M$ such that each boundary slope of $R'$ in $M$ has length within $\epsilon$ of the corresponding length on $\partial H$ in $R$.
\item Dehn filling $M$ along the boundary slopes of $R'$ results in a reducible manifold.
\end{enumerate}
\end{theorem}

\begin{proof}
Take the manifold $M$ of Theorem~\ref{Thm:Cusped}, with embedded surface $R_0:=R'$ in the neighborhood of $R$ in the moduli space, and cusp neighborhood satisfying the requirements of the theorem. 
When we Dehn fill along boundary slopes of $R'$, we attach a collection of disks to $R'$, capping off the punctures to obtain a sphere $S$. Dehn filling along a single cusp attaches a disk to $R'$ that meets the core $c$ of the Dehn filling solid torus exactly once. Because each cusp of $M$ meets at most one cusp of $R'$, it follows that after all Dehn fillings, the core $c$ meets the sphere $S$ exactly once. Hence $S$ cannot bound a ball to either side, and so it is a reducing sphere for the Dehn filling.
\end{proof}


\section{Packings of planar surfaces}
In light of the previous section, in order to find 3--manifolds with reducible fillings with long slopes, it suffices to find horoball packings of planar surfaces where all horoballs are above a given length. In this section, we construct horoball packings of a planar surface where each horoball has length at least $10/\sqrt{3}-\epsilon$ (with $\epsilon > 0$). Our construction relies on a limiting argument, and so the bound of $10/\sqrt{3}$ is only realized asymptotically.

The remainder of this section will be dedicated to proving the following proposition. 

\begin{prop}\label{prop:planar_realization}
For any $\epsilon > 0$, there is a planar surface $F$ with hyperbolic structure and an embedded horoball neighborhood of the cusps of $F$ such that each cusp neighborhood has area $A$ with $A > 10/\sqrt{3}-\epsilon$. 
\end{prop}

\subsection{The triangulated surfaces}
First, we construct a sequence $\{(F_m, T_m)\}$, where $F_m$ is a planar surface (later decorated with a specific horoball packing), and $T_m$ is a triangulation of $F_m$. 

To begin consider the icosahedron, which gives a triangulation 
of the sphere with 12 vertices (each of degree five), 30 edges, and 20 triangular faces. This also gives an ideal triangulation of the 12--punctured sphere where each ideal vertex is degree 5. We call this triangulation $T_0$ and the corresponding surface $F_0$. Initially, we will consider $F_0$ as a topological surface, namely a 12--punctured sphere. Later, we will decorate $F_0$ with geometric data. We will consider $(F_0,T_0)$ as the starting point of a sequence of triangulations of planar surfaces.

We obtain $(F_m, T_m)$ from $(F_0, T_0)$, for $m\geq 1$, as follows. 
For each triangular face of $(F_0, T_0)$, partition each of the three edges of the triangle into $2m$ pieces by adding $2m-1$ new vertices in the interior of the edge. Then, connect the vertices along edges by lines parallel to the sides of the triangle, as in Figure~\ref{fig:subdivideTriangle}. This adds $4m^2$ triangles to each triangular face of $(F_0, T_0)$, giving a new triangulation of $S^2$. Note twelve of the vertices remain degree five, but all additional vertices have degree six. Again make each vertex an ideal vertex, and denote the result by $(F_m, T_m)$; the surface $F_m$ is a punctured sphere.
For $m\geq 1$, the resulting triangulation $T_m$ will have $80m^2$ triangular faces, $40m^2+2$ vertices and $120m^2$ edges.

\begin{figure}
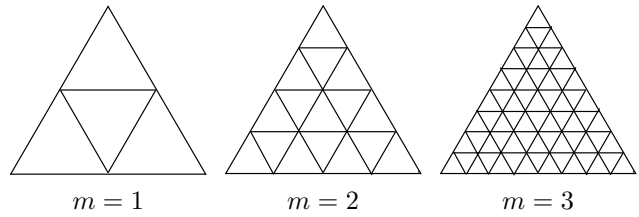
\caption{Subdivisions of triangles making up the icosohedron.}
  \label{fig:subdivideTriangle}
\end{figure}

\subsection{Local packings of the surfaces}
We now decorate the triangulation $T_m$ by assigning a cusp area to each ideal vertex of each ideal triangle. We denote by $[a,b,c]$ a packing of an ideal triangle with cusps of areas $a,b,c$.
If a cusp has area $a$, then the triangle is isometric to a triangle in $\HH^2$ with vertices $0$, $a$, and $\infty$, with a cusp about infinity of height $1$ (hence area $a$). Cusps with areas $b$ and $c$ are mapped into horoballs centered at $0$ and $a$, respectively. See Figure~\ref{fig:triangles}(a).

\begin{figure}
\begin{tabular}{ccc}
  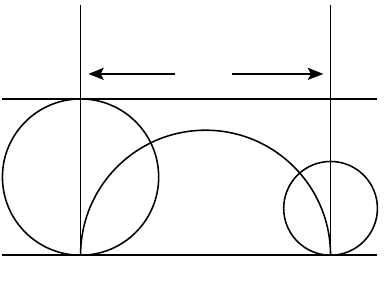 & \hspace{.2in} &  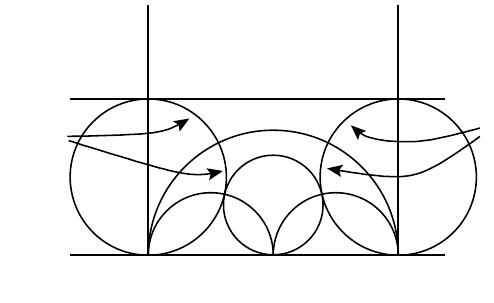 \\
  (a) & & (b) \\
\end{tabular}
  \caption{(a) Ideal triangle with vertices $0$, $a$, $\infty$, and cusp areas $a$, $b$, $c$. (b) Ideal quadrilateral made of two triangles with decorations $[a, 1/a, 1/a]$.}
  \label{fig:triangles}
\end{figure}

\begin{define}\label{def:PackingConditions}
A choice of decorations for each triangle of a triangulation $(F, T)$ is said to be \emph{geometric} if the following conditions hold. 
\begin{enumerate}
\item[(A)] No vertex of a triangle is decorated with area more than two.
\item[(B)] If one vertex is decorated with $A$, the remaining vertices are decorated with areas no more than $1/A$.
\item[(C)] If two triangles share an edge in a triangulation, with ideal vertices of that edge labeled $a_1$ and $b_1$ in one triangle, and $a_2$ and $b_2$ in the other, then $a_1b_1 = a_2b_2$. 
\end{enumerate}
\end{define}

Condition (A) ensures that the cusps of the triangle are embedded in the triangle. Condition (B) ensures that the cusps are disjoint in each triangle. Finally, condition (C) ensures that triangles can be glued together across edges of the triangulation $T$, for in that case the edge between the triangles will have the same length outside the horoballs corresponding to cusp areas, as ensured by the following lemma. 

\begin{lemma}\label{lemma:DistBetweenCusps}
On an ideal triangle with embedded cusps of area $a$ and $b$, the distance between those cusps is $-\log(ab)$.
\end{lemma}

\begin{proof}
The proof is by elementary hyperbolic geometry. One way to see it is to consider the ideal triangle with vertices $0$, $a$, and $\infty$, with the cusp about $\infty$ of height $1$ (area $a$), and the cusp of area $b$ at $0$. The desired distance is $-\log(h)$ where $h$ is the diameter of the horoball about $0$. Apply the isometry $z\mapsto -h/z$, taking $0$ to $\infty$, $\infty$ to $0$, and $ih$ to $i$. Comparing lengths on the triangle, this must take the point $a$ to $-b$. Thus $-h/a=-b$, and $h=ab$. 
\end{proof}

\begin{lemma}\label{lemma:GeomTriang}
A decoration of $(F,T)$ that is geometric, in the sense of Definition~\ref{def:PackingConditions}, 
determines a complete hyperbolic structure on $F$ with a choice of horoball neighborhood at each cusp, such that horoball neighborhoods are embedded.
\end{lemma}

\begin{proof}
We obtain a hyperbolic structure by giving each ideal triangle of $T$ the structure of a hyperbolic ideal triangle. For each such triangle, take horoball neighborhoods of each ideal vertex that have the areas specified by the decoration. These will give embedded horoballs about the cusps of $F$. We need to check that the resulting structure will be complete. 
Completeness follows from the Poincar\'e theorem (see, e.g.\ \cite{epstein-petronio}) and condition (C). By the Poincar\'e theorem, the structure will be complete provided horoball neighborhoods of the cusps ``close up,'' i.e.\ a curve running once around a cusp lifts to a Euclidean translation along a horocycle. In terms of a triangulation, this means that if we lift the triangles meeting a cusp in a cycle, the initial and final triangle (which are identified by a covering transformation) differ only by a Euclidean translation, with no rescaling. In our case, lift the initial triangle to $\HH^2$ with cusp at $\infty$. Our decoration determines a horoball about $\infty$. The lengths of edges between this cusp and the other two determine (Euclidean) diameters of the other two horoballs. Successively lift other ideal triangles to form a cycle. Condition (C) ensures the horoballs between them are consistent. Finally, when we lift the final triangle, because we have fixed the distance between horoballs by (C), and the horoball about $\infty$ has not changed, the diameters of horoballs on $\RR$ must also have the same diameter as in the initial triangle. Thus the initial and final triangles cannot differ by a (non-trivial) scaling, and must differ only by a Euclidean translation. 
\end{proof}

In our construction below, conditions (A) and (B) of Definition~\ref{def:PackingConditions} will be easy to check. For condition (C), for many triangles, we will identify one marked edge of the triangle, such that two triangles are glued across marked edges forming a quadrilateral (with marked edge as the diagonal). We will choose cusp areas on opposite vertices of the quadrilateral to have areas $a$ and $1/a$, respectively, as in Figure~\ref{fig:triangles}(b). This will ensure that the diagonal edge has the same length between horoballs in each triangle. Moreover, exterior edges of the quadrilateral all have length zero. Hence we may glue any two such quads together along an exterior edge consistently.
For the remaining triangles, which will not fit into quadrilaterals, we must check edge lengths between horoballs using the formula of Lemma~\ref{lemma:DistBetweenCusps}. 

We now construct a horoball packing of a planar surface with the desired properties.
 
For each large triangle region in $T_m$ that was formed by replacing one triangle in $T_0$ with $4m^2$ triangles, color the middle $m^2$ triangles in the region gray, as in Figure~\ref{fig:largeTri}(a), and color the remaining triangles white.

\begin{figure}
  \centering
  \begin{tabular}{lclcl}
\resizebox{1.3in}{!}{\includegraphics[width=2in]{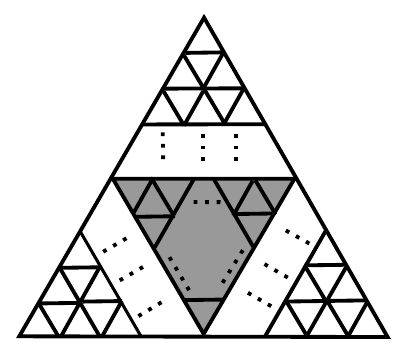}} & &
    \resizebox{1.3in}{!}{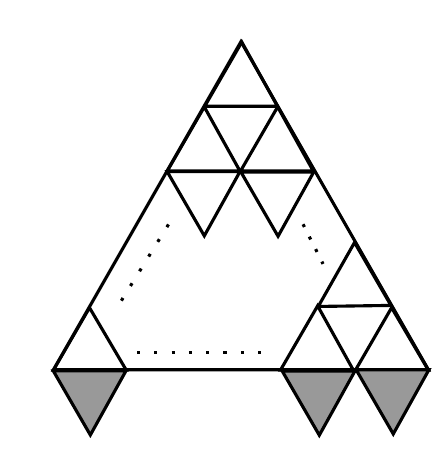} & & 
    \resizebox{1.3in}{!}{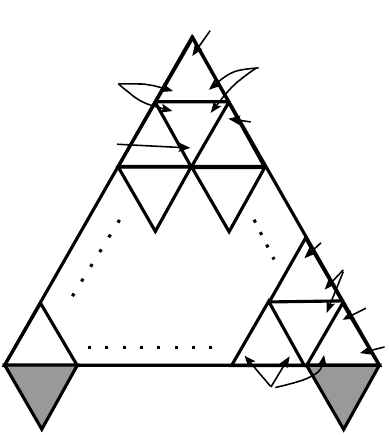} \\
    (a) &  & (b) & & (c)
  \end{tabular}
  \caption[]{(a) Shade the middle triangles gray. (b) Types of vertices. (c)~Decorated packing.
}
  \label{fig:largeTri}
\end{figure}

In order to describe our horoball packing, we will consider six types of vertices in the small triangles appearing in $T_m$, also shown in Figure~\ref{fig:largeTri}(b): 
\begin{enumerate}
\item\label{itm:LargeWhiteCorner} vertices of degree five at the corners of the large triangular regions coming from the triangles of the original icosohedron; 
\item\label{itm:IntWhiteEdge} vertices along the edges of the white triangular regions, thus lying on an edge of the original triangle, but in the interior of that edge;
\item\label{itm:IntWhite} interior vertices in the white triangular regions;
\item\label{itm:GrayCorner} vertices where four white triangular regions and two gray triangular regions are identified;
\item\label{itm:IntGrayEdge} vertices on the edges of a gray triangular region, but in the interior of the edge; and
\item\label{itm:IntGray} vertices in the interior of the gray triangular regions. 
\end{enumerate}

The proof will assign a cusp area to each cusp of each triangle in $T_m$ depending on the pattern of the vertices in that triangle.

\begin{lemma}\label{lemma:TriangulationConstruction}
For any $m$, the triangulated surface $(F_m, T_m)$ constructed above admits a decoration depending on a choice of positive areas $\{c_1, \dots, c_m\}$ such that, if $1\leq c_i \leq 2$, then the decoration gives a complete hyperbolic structure on $F_m$ and embedded cusp neighborhoods with areas of the cusp neighborhoods satisfy the following. 
\begin{itemize}
\item On the 12 vertices of type \eqref{itm:LargeWhiteCorner}, the area is $L_m = 5\,c_m$.
\item On vertices of type \eqref{itm:IntWhiteEdge}, the area is $L_j = 4/c_{j+1} + 2c_j$, for $1\leq j<m$. 
\item On vertices of type \eqref{itm:IntWhite}, the area is $L_j' = 4/c_{j+1} + c_j +c_{j+2}$ for $1 \leq j < m-1$.
\item On vertices of type \eqref{itm:GrayCorner}, the area is $2/c_1^2 + 4/c_1$.
\item On vertices of type \eqref{itm:IntGrayEdge}, the area is either $c_2 + 2 + 3/c_1$, or $c_2+1+4/c_1$. 
\item On vertices of type~\eqref{itm:IntGray}, the area is $2\,c_1 + 4$, or $c_1+5$, or $6$.
\end{itemize}
\end{lemma}

\begin{proof}
Consider first one of the three white triangular regions making up the original larger triangle of the icosohedron. Let $v$ be a vertex coming from the original triangle, of type~\eqref{itm:LargeWhiteCorner}, and let $T$ denote the (small) triangle of $T_m$ that lies in this triangular region and that has $v$ as a vertex. Inside $T$, decorate $v$ with $c_m \geq 1$. The value $c_m$ will correspond to the area of the cusp in $T$. Note that with this decoration in all triangles adjacent $v$, the cusp of $F_m$ corresponding to $v$ will have area $L_m = 5 c_m$. 

Now consider the other vertices of the (small) white triangle $T$. Decorate the two remaining vertices with $1/c_m$. 
Across the edge of $T$ opposite $v$, apply the mirror image decoration. Thus $T$ and its adjacent triangle in the large white triangular region will both have packing $[c_m,1/c_m,1/c_m]$, as in Figure~\ref{fig:largeTri}(c).

Next consider the edges of the large white triangular regions, and the vertices in the interior of this edge, i.e.\ those of type~\eqref{itm:IntWhiteEdge}. For ease of explanation, rotate the triangular region such that the vertex $v$ lies at the top. Then (if $m>1$) the vertex nearest $v$ belongs to exactly one small white triangle in this triangular region that has not yet been labeled, and this is the top-most triangle lying along this edge with top vertex in the interior of the edge. Decorate the vertices of this triangle with $[c_{m-1},1/c_{m-1},1/c_{m-1}]$, and decorate the triangle sharing the bottom edge of this triangle with the reflected packing. These two triangles together form a diamond, with top and bottom vertices labeled $c_{m-1}$.

Now consider the ``diamonds'' just below this top-most diamond, at height $m-2$. Label the top triangle in the diamond by $[c_{m-2}, 1/c_{m-2}, 1/c_{m-2}]$, with $c_{m-2}$ at the top, and assign labels to the bottom triangle by reflecting across the common edge of the two triangles in the diamond. Inside the large white triangular region, there are exactly three diamonds at height $m-2$. Label all three of them in this manner, as in Figure~\ref{fig:largeTri}(c).

Continue, labeling all diamonds at height $k$ with labels $[c_k, 1/c_k, 1/c_k]$. Repeating this process will label all interior vertices of all triangles inside the large white triangular region, with the exception of the triangles across the bottom, sharing an edge with a gray triangle. Label these triangles $[c_1, 1/c_1, 1/c_1]$, with $c_1$ at the top, as in Figure~\ref{fig:largeTri}(c). 

Using this labeling, a vertex at height $j$ of type~\eqref{itm:IntWhiteEdge}, i.e.\ along an edge of the white triangular region in the interior of the edge, will glue up to a cusp with area $L_j=(4/c_{j+1})+2 c_{j}$ for $1\leq j < m$. A vertex of type~\eqref{itm:IntWhite}, in the interior of the white triangular region, will have area $L'_j=(4/c_j)+c_{j+1}+ c_{j-1}$ for $2\leq j < m$.

Now we label gray triangles. For one of the three gray triangles in the corners of the large gray triangular region, decorate the vertices with $[1, 1, 1/c_1^2]$, such that the cusp labeled $1/c_1^2$ corresponds to a vertex of type~\eqref{itm:GrayCorner}.
This triangle has two sides sharing an edge with a small white triangle, with labels on the endpoints of the edges in those white triangles given by $1/c_1$ and $1/c_1$. 
Note by Lemma~\ref{lemma:DistBetweenCusps} that the length of the edge shared by these white and gray triangles will be $-2\log(1/c_1)$ in both triangles, and so the gluing will satisfy condition~(c) of Definition~\ref{def:PackingConditions}. The total area of the cusp at the vertex of type~\eqref{itm:GrayCorner} will be $4/c_1 + 2/c_1^2$.

For the other gray triangles, if the gray triangle shares an edge with a white triangle, decorate its cusps with $[c_1, 1/c_1, 1/c_1]$, with $c_1$ opposite that edge. For all other gray triangles in the interior, decorate their vertices with $[1,1,1]$.

Thus a vertex of type~\eqref{itm:IntGrayEdge} will correspond to a cusp of area $2+3/c_1 + c_1$ if the vertex shares an edge with a vertex of type~\eqref{itm:GrayCorner}, and area $c_2+4/c_1+1$ otherwise. Vertices of type~\eqref{itm:IntGray} either correspond to cusps with area $2c_1+4$, $c_1+5$, or $6$. 

Note that for all choices of $1 \leq c_j \leq 2$, for $j=1, \dots, m$, this decoration will satisfy conditions (A), (B), and (C) of Definition~\ref{def:PackingConditions} by construction, hence by Lemma~\ref{lemma:GeomTriang}, it can be realized by a complete hyperbolic structure on $F_m$ with embedded cusps. 
\end{proof}

Lemma~\ref{lemma:TriangulationConstruction} is stated in terms of areas of cusps, but recall that for a hyperbolic surface, cusp area is the same as length of boundary curve for a cusp.
We wish to ensure that the smallest of the lengths obtained in Lemma~\ref{lemma:TriangulationConstruction} is as large as possible. We will do so by focusing on the lengths $L_j$. Provided we can show that $c_j \leq c_{j+1}$ for all $j$, this will prove that $L_j'\geq L_j$. We will also check the lengths of $L_m$ and lengths along gray triangles.  



\begin{lemma}\label{lemma:cjChoice}
For the decoration of Lemma~\ref{lemma:TriangulationConstruction}, for any $\epsilon>0$, there exists $m$ and a choice of constants $\{c_1, \dots, c_m\}$ such that each $c_i$ satisfies $1\leq c_i \leq 2$, such that $c_i \leq c_{i+1}$ for $i=1, \dots, m-1$, and such that each length along a cusp is at least $10/\sqrt{3}-\epsilon$.
\end{lemma}

\begin{proof} 
Set all of the $L_j$ to be equal to $L$. This gives a system of equations. Considering this system as a discrete dynamical system in the $c_j$, we obtain the recursive formula
\[ c_{j+1}=\frac{4}{L-2c_{j}}.\]
Notice that this defines $c_{j+1}$ in terms of a M\"obius transformation fixing the upper half plane $\HH^2$. Any initial choice of $c_1$ determines all following $c_j$. Moreover, the condition $5\,c_m = L_m=L$ enables us to compute values for the system. As $j$ approaches $\infty$, the numbers $c_j$ will approach an attracting fixed point. That is, they will approach $x$ such that
\[ x = \frac{4}{L-2x}. \]
Appealing to the condition that $L_m = 5x$, we obtain $x=4/(3x)$, indicating that $x=\pm 2/\sqrt{3}$. Using the positive root, we obtain $L=10/\sqrt{3}$. Furthermore, for this choice of $L$, $2/\sqrt{3}$ is the attracting fixed point, and $\sqrt{3}$ the repelling fixed point, of the dynamical system given by the $c_j$'s.
Thus the M\"obius transformation is a hyperbolic translation of the upper half plane, translating along the axis from $\sqrt{3}$ to $2/\sqrt{3}$, and so for $c_1<2/\sqrt{3}$, the sequence $\{c_1, c_2, \dots\}$ will increase monotonically to approach $2/\sqrt{3}$.

Therefore, for any pair $(c_1, \epsilon)$ such that $c_1 < 2/\sqrt{3}$ and 
$\epsilon >0$, there exists an $m$ such that $c_m > 2/\sqrt{3}-\epsilon/5$, 
$L_m > 10/\sqrt{3}-\epsilon$ and $c_j \leq c_{j+1}$.

Thus, given $\epsilon$, any such $c_1$ will ensure lengths of cusps around vertices of types~\eqref{itm:LargeWhiteCorner}, \eqref{itm:IntWhiteEdge}, and \eqref{itm:IntWhite} are at least $10/\sqrt{3}-\epsilon$. We need to ensure lengths of cusps around other vertices are also long. 
Solving $2/c_1^2 + 4/c_1 = L = 10/\sqrt{3}$ for $c_1$, we obtain
\[c_1 = \frac{\sqrt{3}+\sqrt{3+5\sqrt{3}}}{5}\approx 1.029353 < \frac{2}{\sqrt{3}}.\] 
Thus if we choose this value of $c_1$, the vertices of type~\eqref{itm:GrayCorner} will have area exactly $10/\sqrt{3}$. 
A vertex of type~\eqref{itm:IntGrayEdge} will have area either $c_2+2+3/c_1 \geq c_1+2+3/c_1 \approx 5.944 > 10/\sqrt{3}$, or $c_2+1+4/c_1 \geq c_1+1+4/c_1 \approx 5.915 > 10/\sqrt{3}$. 

Finally, vertices of type~\eqref{itm:IntGray} either correspond to cusps with area $2c_1 +4$, $c_1+5$ or $6$. Since $c_1>1$, these cusps are all at least area $6$.
\end{proof}

\begin{proof}[Proof of Proposition~\ref{prop:planar_realization}]
Given $\epsilon>0$, select $m$ and $\{c_1, \dots, c_m\}$ such that the triangulation and decoration of Lemma~\ref{lemma:TriangulationConstruction} satisfy the conclusions of Lemma~\ref{lemma:cjChoice}. 
\end{proof}
 
The above packing has the seemingly undesirable property that the average area of a cusp in the packing is not equal to the minimum area of a cusp in the packing.
Thus, we might hope to shrink the area of the cusps with area above $10/\sqrt{3}$ and use that room to increase the smallest cusp areas in the packing. However, we were unable to find such a packing that could increase the areas around the vertices of type~\eqref{itm:LargeWhiteCorner} to a value above $10/\sqrt{3}$. 

\subsection{Proof of Theorem~\ref{thm:main}.}
We now prove the main theorem. 

\begin{proof}[Proof of Theorem \ref{thm:main}]
Proposition~\ref{prop:planar_realization} shows for all $\epsilon' > 0$ there exists a horoball packing $H$ of a planar surface $F$ such that all cusps have area at least $10/\sqrt{3}-\epsilon'$.
For any $\epsilon''>0$, Theorem~\ref{Thm:ReducibleFilling} provides a cusped hyperbolic 3--manifold $M$, a horoball packing $H_M$ of $M$, and a topological embedding of $F$ into $M$ such that boundary slopes of the image of $F$ have length at least $10/\sqrt{3}-\epsilon'-\epsilon''$ on $\partial H_M$.
To complete the proof, choose $\epsilon'$ and $\epsilon''$ such that $\epsilon > \epsilon'+\epsilon''$. By Theorem~\ref{Thm:ReducibleFilling}, the Dehn filling along those boundary slopes is a reducible manifold. 
\end{proof}


\section{Graph manifolds and Seifert fibered spaces}

In this section, we will consider lengths of slopes giving exceptional fillings besides reducible. 
Refining our discussion from the introduction, an exceptional filling $s$ of a one--cusped hyperbolic manifold $M$ is either
reducible, finite (i.e. $|\pi_1(M(s))| < \infty$), small Seifert fibered with infinite fundamental group, or toroidal. We treat each case separately.

\subsection{Toroidal fillings and fillings along torus slopes}

We first consider lengths of slopes of manifolds that admit a toroidal filling, or more generally, lengths of slopes that are parallel to the boundary components of an essential punctured torus. As noted in the introduction, in the toroidal case, the bound of six is realized. Agol has examples in \cite{agol:bounds}, as does Adams \emph{et al} for knots in $S^3$ \cite{adams:TotallyGeodesic}. Both of these examples involve a specific packing of an embedded totally geodesic once--punctured torus. In each case, these packings are as dense as possible. In fact, this behavior is required for a toroidal filling realizing the bound of six, due to the following lemma, which follows quickly from the proof of \cite[Theorem~5.1]{agol:bounds}. 

\begin{lemma}\label{lemma:six-realized}
Let $M$ be a cusped hyperbolic 3--manifold and let $S$ be an essential punctured torus properly immersed in $M$. Let $s_1, \dots, s_n$ be slopes on the cusps of $M$ homotopic to the boundary slopes of $S$. Then there exists a cusp neighborhood $C$ for $M$ such that each $s_j$ has length exactly $6$ on $\bdy C$ if and only if the following hold:
\begin{enumerate}
\item $S$ is totally geodesic,
\item $H=C \cap S$ is a maximal cusp of $S$, and
\item the canonical decomposition of $S$ corresponding to $H$ consists of ideal triangles, each with a $[1,1,1]$ packing given by its intersection with $H$.
\end{enumerate}
\end{lemma}

\begin{proof}
In \cite[Theorem~5.1]{agol:bounds}, Agol proves that the sum of the lengths of the slopes $s_j$ is at most $6|\chi(S)|$. We walk through the proof of that theorem to determine what must occur when we have equality. First, pleat $S$ in $M$, and consider $S\cap C$, shrinking $C$ if necessary to $C'$ to avoid compact regions of intersection. The lengths of the slopes $s_i$ on $\bdy C'$ in $M$ are at most the lengths of the curves running along the cusp $C'\cap S$ of $S$ in the hyperbolic structure it inherits from the pleating. We have equality in this case if and only if there is no bending along the pleating locus, i.e.\ the surface $S$ is totally geodesic.

Next, the length of a curve along $H'=C'\cap S$ is known to be the area of $H'$. By work of B\"or\"oczky \cite{boroczky}, this area will be at most $3/\pi$ times the area of $S$. We will have equality if and only if the cusp $H'$ is as dense as possible in $S$, which will happen if and only if the canonical decomposition of $S$ with respect to $H'$ consists of triangles, each with packing $[1,1,1]$. Since $C'$ cannot be expanded without $H'=C'\cap S$ becoming no longer embedded, $H'$ must equal $H$, and the lemma follows. 
\end{proof}

\begin{remark}
Suppose $S$ has a horoball packing $H$ that is as dense as possible as in B\"or\"oczky~\cite{boroczky}, and suppose we have an ideal triangulation of $S$. It does not follow that each triangle has a $[1,1,1]$ packing. For example, a triangle might have a $[2, 1/2, 1/2]$ packing. However, these triangles are not canonical. The canonical packing is shown in Figure~\ref{fig:localHoroballPacking}. 
\end{remark}

\begin{figure}
\includegraphics[height=3.8cm]{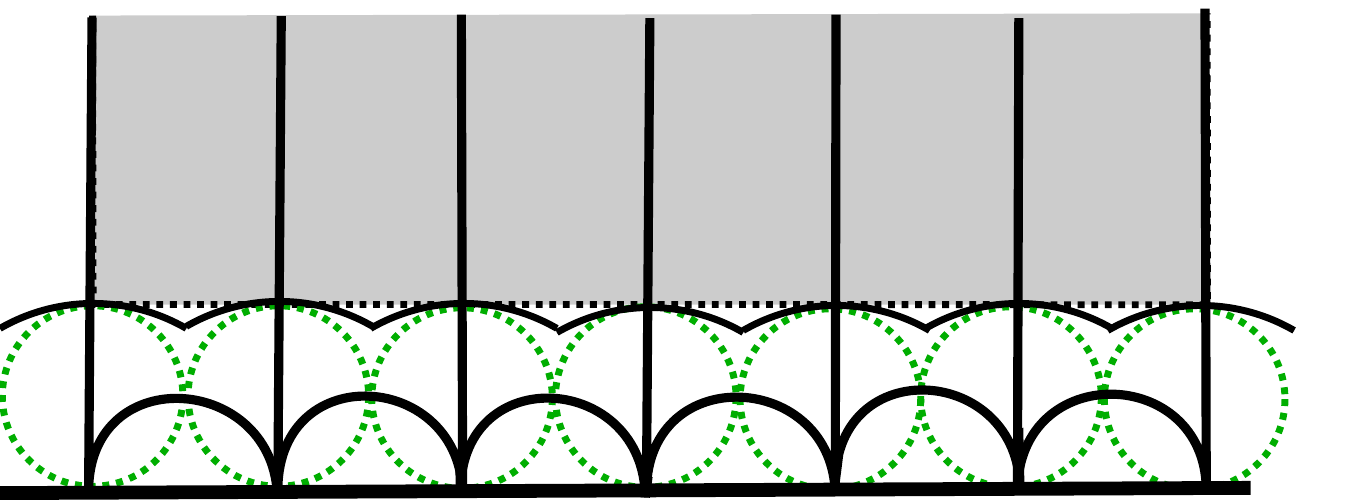}
\caption{The local picture of a horoball packing of $F$. Each horoball tangent to the horoball at $\infty$ is diameter one. The centers of the curvilinear triangles in the complement of the horoball packing are at height $\sqrt{3}/{2}$.}
\label{fig:localHoroballPacking} 
\end{figure}

This packing will be a key tool used to address the slope lengths of Dehn fillings that produce Seifert fibered manifolds with infinite fundamental group. 

\subsection{The Seifert fibered case}

Next we consider examples of manifolds with long slopes yielding small Seifert fibered fillings with infinite fundamental group. The following proposition gives a lower bound on the longest such slope length.

\begin{prop}\label{prop:SFlength5}
There exists a sequence of hyperbolic 3--manifolds $M_n$, each with one cusp, and slopes $s_n$ with length approaching $5$ such that $M_n(s_n)$ is a small Seifert fibered manifold with infinite fundamental group.
\end{prop}

\begin{proof}
Consider the Whitehead Sister manifold, which is $m125$ in the SnapPy census~\cite{snappy}. Our numbering and framing of cusps of this manifold will be consistent with SnapPy.

Choose a horoball packing of $m125$ such that cusp~0 touches itself. We first find a slope on cusp~0 of length $5$ that yields a manifold with interesting properties. 

Consider a fundamental domain for $m125$ such that the point of tangency of cusp~0 occurs at height~1. Then there exists a representation $\rho\co \pi_1(m125) \to \PSL(2,\CC)$ such that cusp~0 lifts to $\infty$ in the universal cover, and so an element of the peripheral subgroup at $\infty$ can be realized as a translation in the plane at height~1. If we identify this plane with the complex plane, we can further assume (after conjugating by a rotation fixing $\infty$ if necessary) that the slope $(1,0)$ corresponds to translation by $2-i$ and the $(0,1)$ slope can be realized as a translation by $1+2i$.
Under this representation, the $(2,1)$ slope on cusp~0 is realized by a translation by $4+3i$. Hence it has length $5$. Using Regina~\cite{burton:regina}, one can identify explicitly the manifold obtained by Dehn filling this slope; it is the one-cusped graph manifold
\[
M'\cong D^2(2,1)(3,1) \cup_{\tiny{ \begin{pmatrix} 1 & 1\\ 1 & 0 \end{pmatrix} } } A(2,1).
\]
(This result can also be verified by~\cite{martellipetronio:magic}, where the notation is explained, but it will not be important in our argument. Note also that the framing of SnapPy has changed since that paper was published; adjust their framings on the Whitehead Sister manifold in~\cite[Table~A.1]{martellipetronio:magic} by $+2$ to obtain the current framing.)

Now consider cusp~1 on $m125$. Another feature of this manifold is that there is a 1--parameter family of fillings $(r,1)$  of cusp~1 such that
\[
m125()(r,1) \cong S^2(2,1)(3,1)(2r+1,-2r+1)) -K_r,
\]
for some knot $K_r$ with meridian corresponding to the $(2,1)$ slope of cusp~0 on $m125$, hence
\[
m125(2,1)(r,1) \cong S^2(2,1)(3,1)(2r+1,-2r+1)).
\]
This is a small Seifert fibered space. Its fundamental group will be infinite provided $r$ is not one of $-3,-2,-1,0,1,2$.
Thus, for sufficiently large $r$, the cusped manifolds 
\[
M_r = m125()(r,1) \cong S^2(2,1)(3,1)(2r+1,-2r+1)) -K_r
\]
approach m125, and so are hyperbolic by Thurston's Hyperbolic Dehn Surgery Theorem~\cite[5.8.2]{thurston}. The slope on $M_r$ corresponding to the slope $(2,1)$ on cusp~0 of $m125$ has length approaching its maximal length in $m125$, which, as mentioned above, is $5$.
\end{proof}

Proposition~\ref{prop:SFlength5} gives a lower bound of five on the length of the longest slope yielding a small Seifert fibered space with infinite fundamental group. We do not claim this is a sharp lower bound. However, we can prove that no slope corresponding to a small Seifert fibered space has length exactly six.  

\begin{theorem}\label{thm:SSFbound}
Let $M$ be a hyperbolic manifold with one cusp, and let $s$ be a slope such that $M(s)$ is a small Seifert fibered space with infinite fundamental group.  Then the length of $s$ is strictly less than six.  
\end{theorem}

\begin{proof}
Suppose $M(s)$ is a small Seifert fibered space with infinite fundamental group, but the length of $s$ is at least six.  The 6--Theorem implies the length is exactly six.

The base orbifold of $M(s)$ is a sphere with three cone points.  Work of Scott~\cite{scott1983noFake} implies that there is an immersed essential torus in $M(s)$.  Because $M$ is hyperbolic, there must be an immersed essential punctured torus $T$ in $M$, punctured $n$ times.  Thus $T$ meets $\partial M$ $n$ times, each with slope $s$.

We may use the hyperbolic structure of $M$ to give a pleating on $T$ in $M$. Since the length of each slope is exactly six, the immersed surface must be totally geodesic, and the horoball packing on each cusp must be the densest packing possible, by Lemma~\ref{lemma:six-realized}.
Thus in the universal cover $\HH^3$, $T$ lifts to a geodesic plane $\bar{T}$ and the maximal cusp for $M$ lifts to the densest possible horoball packing on this plane. That is, horoballs meet in triples with triangular regions in the complement.

Now, because $T$ is merely immersed and not embedded, there must be another lift $\tilde{T}$ of $T$ to $\HH^3$ that intersects $\bar{T}$.  Note $\tilde{T}$ must also be totally geodesic, hence $\bar{T} \cap \tilde{T}$ is a geodesic in $\HH^3$ on the vertical plane $\bar{T}$.  First, note that $\bar{T} \cap \tilde{T}$ cannot be a vertical geodesic, for then $\bar{T}$ and $\tilde{T}$ would be nonparallel planes, defining distinct slopes on $\partial M$.  So $\bar{T} \cap \tilde{T}$ is a geodesic $\gamma \in \HH^3$ with endpoints disjoint from the centers of horoballs in the packing on $T$.

We claim $\gamma$ must intersect some horoball $H_0$ in a compact interval with nonempty interior.  This is because, first, the endpoints of $\gamma$ are disjoint from horoballs. Next, the complement of all horoballs in $T$ is a collection of curvilinear triangles.  If $\gamma$ meets no such $H_0$ then it meets a triangle.  It cannot meet it in its edges, which are boundaries of horoballs, since then $\gamma$ would continue through the interior of the corresponding horoball.  But if it meets the triangle only in vertices, then it must run along the geodesic between those vertices, which runs through the interior of a horoball.  

So $\gamma$ must intersect some horoball $H_0$ in a compact interval with nonempty interior.  It follows that $\tilde{T}$ must intersect $H_0$ in a compact disk.  Since $\tilde{T}$ differs from $\bar{T}$ by a covering isometry, there must be some horoball lift of the embedded maximal cusp $H_1\subset \HH^3$ such that $H_1 \cap \bar{T}$ is a compact disk. We will show this leads to a contradiction.

Since the maximal cusp of $M$ is embedded, all lifts of the cusp to $\HH^3$ are disjoint.  Thus if $H_1 \cap \bar{T}$ is a compact disk, it must lie in one of the triangular regions in the complement of the horoballs with centers on $\partial \bar{T}$.  In particular, the disk must have a center closest to some horoball $H_2$ on $\bar{T}$.  Applying a hyperbolic isometry, we may assume that $H_2$ is the horoball at infinity, of height one, and the disk of intersection $H_1 \cap \bar{T}$ lies in one of the small triangles between the horoball at infinity and two tangent horoballs of diameter one.  Denote the Euclidean center of the disk $H_1 \cap \bar{T} \cong \HH^2$ by $q$.  The height of $q$ is also the Euclidean radius of $H_1$.  But $q$ lies in a triangular region above horoballs of diameter one, hence it has height at least $\sqrt{3}/2$.  But then the diameter of $H_1$ is at least $\sqrt{3} > 1$, contradicting the fact that $H_1$ is disjoint from the horoball at infinity. This contradiction concludes the proof that the length of the slope $s$ is strictly less than six.
\end{proof}

\begin{remark}
The above proof applies to any space with a Dehn filling yielding a manifold with an immersed essential torus. For example, this will apply to many graph manifolds as well. The longest graph manifold slope for one--cusped manifolds in the SnapPy census has length $2\sqrt{5\sqrt{2}} \approx 5.318$, and occurs for the manifold {\tt{m130}}. It would be interesting to find a sharp upper bound for slope length for graph manifold fillings. 

To put this into a better context, the example in \cite{agol:bounds} has eight tetrahedra and the first example of \cite{adams:TotallyGeodesic} has nine tetrahedra, and so both manifolds appear in the census with nine or fewer tetrahedra. While we expect to see some very interesting behavior regarding exceptional surgeries in the cusped census, not all long slopes appear this way. For example, the reducible fillings of the main theorem have too many cusps to appear in any existing census.
\end{remark}

\subsection{Finite fillings}

We say only a few words on what is known in the case of finite fillings, and conclude this section with a summary. 

There are many examples of one--cusped hyperbolic manifolds with slopes that yield $S^3$ fillings with slope lengths approaching $4$, for example \cite[Section~7]{agol:bounds} and \cite[Section~4]{purcell:slope}. Note the same basic construction of used in the proof of \cite[Corollary~4.4]{purcell:slope} can also produce hyperbolic knots in lens spaces with meridians of length approaching 4. The final step of that argument is to fill an unknotted component of a link by a slope of $1/c_2$. However, if we allow $p/c_2$ fillings with $p>1$ and $p$ and $c_2$ relatively prime, then the knot complements constructed still have meridian lengths approaching 4, but will be knot complements in (non-trivial) lens spaces as opposed to in $S^3$. A similar argument can also be applied to the knot complements constructed in \cite[Section~7]{agol:bounds}.

If we allow multiple cusps, work of Goerner shows that $4$ is not an upper bound on slope length. In \cite[Table~1]{goerner2014regular}, Goerner exhibits an arithmetic construction of a link in $S^3$ with 64 cusps such that each meridian is at least length $\sqrt{21}\approx 4.5826$.

Notice that this difference between known lengths of slopes for one-cusped versus multi-cusped manifolds also appears in the reducible case. The examples of Theorem~\ref{thm:main} prove that the maximal length reducible slope in the multi-cusped case is at least $10/\sqrt{3}$. However, when we restrict to one-cusped manifolds, the longest length examples we could find were reducible fillings obtained from a sequence of fillings of the minimally--twisted 5-chain link, with slope lengths approaching $4$. Details on fillings of the minimally--twisted 5-chain link can be found in~\cite{martellipetronioroukema}.

\begin{table}[t]
\begin{tabular}{c c c }
Type & one-cusp & multi-cusp \\
\midrule
Finite & 4* & $\sqrt{21}$\\
Reducible & 4* & $10/\sqrt{3}*$\\
small SFS  & 5* & 5 \\
Toroidal & 6 & 6 \\
\end{tabular} 
\caption{\label{tab:LongSlopes}  A table of the lengths of observed slopes that realize each type of exceptional filling. Slope lengths marked with an asterisk are those for which the bound is realized asymptotically. 
  The lengths of the toroidal fillings are the only ones known to be sharp.  } 
\end{table}

Table~\ref{tab:LongSlopes}  provides a summary of these examples and the other examples given in this paper. In the table, ``small SFS'' refers to small Seifert fibered space fillings with infinite fundamental group. An asterisk indicates that the bound is realized asymptotically.

We conclude by asking the following questions motivated by the table.

\begin{question}
If $M$ is a one-cusped hyperbolic manifold and $s$ is a reducible, finite, or small Seifert fibered slope, what is the maximal length of $s$? 
\end{question}

\begin{question}
If $M$ is a multi-cusped hyperbolic manifold and $s$ is a reducible, finite, or small Seifert fibered slope, what is the maximal length of $s$?
\end{question}

\bibliographystyle{amsplain}
\bibliography{references}

\end{document}

%% file: Figures/triangleSubdivide.pdf_t
\begin{picture}(0,0)%
\includegraphics{Figures/triangleSubdivide.pdf}%
\end{picture}%
\setlength{\unitlength}{3947sp}%
\begingroup\makeatletter\ifx\SetFigFont\undefined%
\gdef\SetFigFont#1#2#3#4#5{%
  \reset@font\fontsize{#1}{#2pt}%
  \fontfamily{#3}\fontseries{#4}\fontshape{#5}%
  \selectfont}%
\fi\endgroup%
\begin{picture}(3950,1346)(577,-1241)
\put(2326,-1186){\makebox(0,0)[lb]{\smash{{\SetFigFont{10}{12.0}{\familydefault}{\mddefault}{\updefault}{\color[rgb]{0,0,0}$m=2$}%
}}}}
\put(3676,-1186){\makebox(0,0)[lb]{\smash{{\SetFigFont{10}{12.0}{\familydefault}{\mddefault}{\updefault}{\color[rgb]{0,0,0}$m=3$}%
}}}}
\put(976,-1186){\makebox(0,0)[lb]{\smash{{\SetFigFont{10}{12.0}{\familydefault}{\mddefault}{\updefault}{\color[rgb]{0,0,0}$m=1$}%
}}}}
\end{picture}%

%% file: Figures/hyperbolicTriangle.pdf_t
\begin{picture}(0,0)%
\includegraphics{Figures/hyperbolicTriangle.pdf}%
\end{picture}%
\setlength{\unitlength}{3947sp}%
\begingroup\makeatletter\ifx\SetFigFont\undefined%
\gdef\SetFigFont#1#2#3#4#5{%
  \reset@font\fontsize{#1}{#2pt}%
  \fontfamily{#3}\fontseries{#4}\fontshape{#5}%
  \selectfont}%
\fi\endgroup%
\begin{picture}(1854,1437)(814,-1186)
\put(1747,-144){\makebox(0,0)[lb]{\smash{{\SetFigFont{10}{12.0}{\familydefault}{\mddefault}{\updefault}{\color[rgb]{0,0,0}$a$}%
}}}}
\put(1289,-402){\makebox(0,0)[lb]{\smash{{\SetFigFont{10}{12.0}{\familydefault}{\mddefault}{\updefault}{\color[rgb]{0,0,0}$b$}%
}}}}
\put(2237,-504){\makebox(0,0)[lb]{\smash{{\SetFigFont{10}{12.0}{\familydefault}{\mddefault}{\updefault}{\color[rgb]{0,0,0}$c$}%
}}}}
\put(1106,-1131){\makebox(0,0)[lb]{\smash{{\SetFigFont{10}{12.0}{\familydefault}{\mddefault}{\updefault}{\color[rgb]{0,0,0}$0$}%
}}}}
\put(2319,-1111){\makebox(0,0)[lb]{\smash{{\SetFigFont{10}{12.0}{\familydefault}{\mddefault}{\updefault}{\color[rgb]{0,0,0}$a$}%
}}}}
\put(2653,-211){\makebox(0,0)[lb]{\smash{{\SetFigFont{10}{12.0}{\familydefault}{\mddefault}{\updefault}{\color[rgb]{0,0,0}$1$}%
}}}}
\end{picture}%

%% file: Figures/hyperbolicQuad.pdf_t
\begin{picture}(0,0)%
\includegraphics{hyperbolicQuad.pdf}%
\end{picture}%
\setlength{\unitlength}{3947sp}%
\begingroup\makeatletter\ifx\SetFigFont\undefined%
\gdef\SetFigFont#1#2#3#4#5{%
  \reset@font\fontsize{#1}{#2pt}%
  \fontfamily{#3}\fontseries{#4}\fontshape{#5}%
  \selectfont}%
\fi\endgroup%
\begin{picture}(2352,1437)(490,-1186)
\put(505,-463){\makebox(0,0)[lb]{\smash{{\SetFigFont{10}{12.0}{\familydefault}{\mddefault}{\updefault}{\color[rgb]{0,0,0}$1/a$}%
}}}}
\put(1106,-1131){\makebox(0,0)[lb]{\smash{{\SetFigFont{10}{12.0}{\familydefault}{\mddefault}{\updefault}{\color[rgb]{0,0,0}$0$}%
}}}}
\put(2319,-1111){\makebox(0,0)[lb]{\smash{{\SetFigFont{10}{12.0}{\familydefault}{\mddefault}{\updefault}{\color[rgb]{0,0,0}$a$}%
}}}}
\put(2653,-211){\makebox(0,0)[lb]{\smash{{\SetFigFont{10}{12.0}{\familydefault}{\mddefault}{\updefault}{\color[rgb]{0,0,0}$1$}%
}}}}
\put(1733,-151){\makebox(0,0)[lb]{\smash{{\SetFigFont{10}{12.0}{\familydefault}{\mddefault}{\updefault}{\color[rgb]{0,0,0}$a$}%
}}}}
\put(1767,-644){\makebox(0,0)[lb]{\smash{{\SetFigFont{10}{12.0}{\familydefault}{\mddefault}{\updefault}{\color[rgb]{0,0,0}$a$}%
}}}}
\put(2827,-423){\makebox(0,0)[lb]{\smash{{\SetFigFont{10}{12.0}{\familydefault}{\mddefault}{\updefault}{\color[rgb]{0,0,0}$1/a$}%
}}}}
\end{picture}%

%% file: Figures/triVertexLabel.pdf_t
\begin{picture}(0,0)%
\includegraphics{triVertexLabel.pdf}%
\end{picture}%
\setlength{\unitlength}{3947sp}%
\begingroup\makeatletter\ifx\SetFigFont\undefined%
\gdef\SetFigFont#1#2#3#4#5{%
  \reset@font\fontsize{#1}{#2pt}%
  \fontfamily{#3}\fontseries{#4}\fontshape{#5}%
  \selectfont}%
\fi\endgroup%
\begin{picture}(2119,2281)(345,-1469)
\put(1742,278){\makebox(0,0)[lb]{\smash{{\SetFigFont{10}{12.0}{\familydefault}{\mddefault}{\updefault}{\color[rgb]{0,0,0}$(2)$}%
}}}}
\put(1425,677){\makebox(0,0)[lb]{\smash{{\SetFigFont{10}{12.0}{\familydefault}{\mddefault}{\updefault}{\color[rgb]{0,0,0}$(1)$}%
}}}}
\put(516,-663){\makebox(0,0)[lb]{\smash{{\SetFigFont{10}{12.0}{\familydefault}{\mddefault}{\updefault}{\color[rgb]{0,0,0}$(2)$}%
}}}}
\put(360,-962){\makebox(0,0)[lb]{\smash{{\SetFigFont{10}{12.0}{\familydefault}{\mddefault}{\updefault}{\color[rgb]{0,0,0}$(4)$}%
}}}}
\put(2449,-968){\makebox(0,0)[lb]{\smash{{\SetFigFont{10}{12.0}{\familydefault}{\mddefault}{\updefault}{\color[rgb]{0,0,0}$(4)$}%
}}}}
\put(951,-1139){\makebox(0,0)[lb]{\smash{{\SetFigFont{10}{12.0}{\familydefault}{\mddefault}{\updefault}{\color[rgb]{0,0,0}$(5)$}%
}}}}
\put(788,-1373){\makebox(0,0)[lb]{\smash{{\SetFigFont{10}{12.0}{\familydefault}{\mddefault}{\updefault}{\color[rgb]{0,0,0}$(5)$}%
}}}}
\put(1061,299){\makebox(0,0)[lb]{\smash{{\SetFigFont{10}{12.0}{\familydefault}{\mddefault}{\updefault}{\color[rgb]{0,0,0}$(2)$}%
}}}}
\put(889,-14){\makebox(0,0)[lb]{\smash{{\SetFigFont{10}{12.0}{\familydefault}{\mddefault}{\updefault}{\color[rgb]{0,0,0}$(2)$}%
}}}}
\put(1176,-485){\makebox(0,0)[lb]{\smash{{\SetFigFont{10}{12.0}{\familydefault}{\mddefault}{\updefault}{\color[rgb]{0,0,0}$(3)$}%
}}}}
\put(1614,-493){\makebox(0,0)[lb]{\smash{{\SetFigFont{10}{12.0}{\familydefault}{\mddefault}{\updefault}{\color[rgb]{0,0,0}$(3)$}%
}}}}
\put(1623,-1409){\makebox(0,0)[lb]{\smash{{\SetFigFont{10}{12.0}{\familydefault}{\mddefault}{\updefault}{\color[rgb]{0,0,0}$(6)$}%
}}}}
\end{picture}%

%% file: Figures/triLabeled.pdf_t
\begin{picture}(0,0)%
\includegraphics{triLabeled.pdf}%
\end{picture}%
\setlength{\unitlength}{3947sp}%
\begingroup\makeatletter\ifx\SetFigFont\undefined%
\gdef\SetFigFont#1#2#3#4#5{%
  \reset@font\fontsize{#1}{#2pt}%
  \fontfamily{#3}\fontseries{#4}\fontshape{#5}%
  \selectfont}%
\fi\endgroup%
\begin{picture}(1860,2083)(579,-1291)
\put(742,329){\makebox(0,0)[lb]{\smash{{\SetFigFont{10}{12.0}{\familydefault}{\mddefault}{\updefault}{\color[rgb]{0,0,0}$1/c_m$}%
}}}}
\put(2424,-929){\makebox(0,0)[lb]{\smash{{\SetFigFont{10}{12.0}{\familydefault}{\mddefault}{\updefault}{\color[rgb]{0,0,0}$1/c_1$}%
}}}}
\put(2331,-721){\makebox(0,0)[lb]{\smash{{\SetFigFont{10}{12.0}{\familydefault}{\mddefault}{\updefault}{\color[rgb]{0,0,0}$c_1$}%
}}}}
\put(2107,-370){\makebox(0,0)[lb]{\smash{{\SetFigFont{10}{12.0}{\familydefault}{\mddefault}{\updefault}{\color[rgb]{0,0,0}$c_2$}%
}}}}
\put(1587,645){\makebox(0,0)[lb]{\smash{{\SetFigFont{10}{12.0}{\familydefault}{\mddefault}{\updefault}{\color[rgb]{0,0,0}$c_m$}%
}}}}
\put(1805,164){\makebox(0,0)[lb]{\smash{{\SetFigFont{10}{12.0}{\familydefault}{\mddefault}{\updefault}{\color[rgb]{0,0,0}$c_{m-1}$}%
}}}}
\put(1766,-1221){\makebox(0,0)[lb]{\smash{{\SetFigFont{10}{12.0}{\familydefault}{\mddefault}{\updefault}{\color[rgb]{0,0,0}$1/c_1$}%
}}}}
\put(1834,424){\makebox(0,0)[lb]{\smash{{\SetFigFont{10}{12.0}{\familydefault}{\mddefault}{\updefault}{\color[rgb]{0,0,0}$1/c_m$}%
}}}}
\put(905, 69){\makebox(0,0)[lb]{\smash{{\SetFigFont{10}{12.0}{\familydefault}{\mddefault}{\updefault}{\color[rgb]{0,0,0}$c_m$}%
}}}}
\put(2248,-551){\makebox(0,0)[lb]{\smash{{\SetFigFont{10}{12.0}{\familydefault}{\mddefault}{\updefault}{\color[rgb]{0,0,0}$1/c_2$}%
}}}}
\end{picture}%